\documentclass[a4paper,10pt]{amsart}
\usepackage{amssymb,amscd}
\usepackage[all]{xy}
\renewcommand{\iff}{if and only if }
\newcommand{\st}{such that }

\newcommand{\modR}{\hbox{{\rm mod-}}R}

\newcommand{\ModR}{\hbox{{\rm Mod-}}R}

\DeclareMathOperator{\Hom}{Hom}
\DeclareMathOperator{\Dom}{Dom}
\DeclareMathOperator{\End}{End}
\DeclareMathOperator{\Ext}{Ext}

\DeclareMathOperator{\Ker}{Ker}
\DeclareMathOperator{\Img}{Im}
\DeclareMathOperator{\Coker}{Coker}

\DeclareMathOperator{\cf}{cf}

\usepackage[usenames]{color}

\theoremstyle{plain}
\newtheorem{thm}{Theorem}[section]
\newtheorem{prop}[thm]{Proposition}
\newtheorem{lem}[thm]{Lemma}

\newtheorem{cor}[thm]{Corollary}
\theoremstyle{definition}
\newtheorem{defn}[thm]{Definition}

\theoremstyle{remark}
\newtheorem{rem}{Remark}

\begin{document}
\title{On the non-existence of right almost split maps}

\author{\textsc{Jan \v Saroch}}
\address{Charles University, Faculty of Mathematics and Physics, Department of Algebra \\ 
Sokolovsk\'{a} 83, 186 75 Praha~8, Czech Republic}
\email{saroch@karlin.mff.cuni.cz}

 

\thanks{This research has been supported by grant GA\v CR 14-15479S}

\date{\today}

\begin{abstract} We show that, over any ring, a module $C$ is a codomain of a~right almost split map \iff $C$ is a finitely presented module with local endomorphism ring; thus we give an answer to a 40 years old question by M. Auslander. Using the tools developed, we also provide a useful sufficient condition for a class of modules to be non-precovering. Finally, we show a non-trivial application in the general context of morphisms determined by object.\end{abstract} 

\maketitle
\vspace{4ex}

\section{Introduction}
\label{sec:intro}

Almost split sequences (also called Auslander--Reiten sequences) represent the central tool of Auslander--Reiten theory. They serve as stepping stones in the hard task to understand the possible extensions in the category of finitely generated modules over an Artin algebra. Their utilization, however, is not restricted to this very context. The theory of almost split sequences is developed in various other categories, for instance, \cite{KL}, in the category of complexes of modules (and correlatively, in its triangle version, in the homotopy category of modules), or for the general case of exact categories, \cite{LNP}.

The important question, common in various contexts, is whether, for a given object $C$, there exists an almost split sequence beginning or ending in $C$ (cf. \cite[questions $(1), (2)$ on pg. 4]{Aus75}). Since the almost split sequence comprises of two parts, the left almost split map and the right almost split map, we can ask even for the mere existence of these maps having $C$ as domain, codomain resp.

In the category $\ModR$, where $R$ is any ring, it is easy to show that a necessary condition on $C$ is that $\End_R(C)$ is local. In fact, by a result of Auslander, a finitely presented module $C$ is the codomain of a right almost split map in $\ModR$ if and only if $\End_R(C)$ is local. There are examples, however, where the domain of the right almost split map is necessarily non-finitely presented. Although it may happen that there still is a right almost split map in $\modR$ having $C$ as its codomain.

These examples illustrate that the relation between the existence of right almost split maps (and correlatively almost split sequences) in categories $\modR$ and $\ModR$ is rather intricate. Suppose, on the other hand, that we want to use the machinery of almost split maps and the properties they provide us with to study modules which are not finitely presented. Is the situation more clear in this case? Could we perhaps obtain some interesting new information on the possible extensions involving a particular countably, or even non-countably presented module?

The answer to these questions, as it turns out, is `yes and no'. To be more precise: the main result of this paper, Theorem~\ref{t:main}, states that every right almost split map has to have a finitely presented codomain. This has been recently conjectured in a slightly weaker form in \cite{Kr}, however, already in \cite{Aus75}, Auslander asked for the precise description of modules appearing as the right-hand terms in almost split sequences.


The proof of our result uses basically two main tools. The first one is a construction of a so-called tree module, i.e., a particular combinatorial object which serves as a test module for splitting of a given epimorphism, in our case of a right almost split map $f:B\to C$ with a $\theta$-presented module $C$ where $\theta$ is an infinite cardinal. The tree module appears as the middle term in a short exact sequence which is subsequently used to show that $C$ has to be, in fact, $<\theta$-presented.

This is done by incorporating the second main tool---a modified version of Hunter's cardinal counting argument (Lemma~\ref{l:Hun}), recently used in \cite{AST} as the key part of the proof that the class of all flat Mittag--Leffler modules is not precovering unless the underlying ring is right perfect.

\smallskip

The structure of this paper is pretty straightforward. After a short section where we fix our notation and recall some of the known results, we describe the construction of tree modules in detail in \S 3. Apart from this rather technical construction, as a sort of byproduct, we answer a question by G.\ Bergman from \cite{Be} in Theorem~\ref{t:Berg}.

In \S 4, we present the main theorem of our paper. We also discuss an application of the machinery at our disposal to the theory of approximations of modules. Finally, in \S 5 we prove a non-trivial result concerning the general concept of morphisms determined by object.

\section{Preliminaries}
\label{sec:prelim}

Throughout the paper, $R$ denotes an (associative) ring with enough idempotents. By a module, we mean a unitary right $R$-module (i.e. $M$ such that $M = MR$). The category of all modules is denoted by $\ModR$, its full subcategory consisting of all finitely presented modules by $\modR$ and the category of all flat modules by Flat-$R$. In this case, $\ModR$ is a finitely accessible Grothendieck category. All our results are theorems in ZFC; bar Lemma~\ref{l:bigtree}, about elements of $\ModR$.

Let $B, C\in\ModR$. A homomorphism $f:B\to C$ is called a \emph{right almost split map} if, given any $M\in\ModR$ and $k\in\Hom_R(M,C)$, the map $k$~factorizes through $f$ \iff $k$ is not a split epimorphism. A \emph{left almost split map} is defined dually. We say that a short exact sequence 
$$0 \longrightarrow A \overset{m}\longrightarrow B \overset{f}\longrightarrow C \longrightarrow 0$$
is an \emph{almost split sequence} if $m$ is left almost split and $f$ is right almost split.

It is not hard to show that a right almost split map appears as the epimorphism in an almost split sequence \iff it is surjective and its kernel has got local endomorphism ring. Moreover, we have the following properties.

\begin{prop}\label{p:few} Let $f:B\to C$ be a right almost split map. Then:
\begin{enumerate}
	\item The endomorphism ring of $C$ is local.
	\item The map $f$ is surjective \iff $C$ is not a projective module.
\end{enumerate}
\end{prop}

The main theorem on the existence of right almost split maps is due to Auslander. It is a partial converse of $(1)$ above.

\begin{thm}\label{t:Aus}{\rm (\cite[Theorem 4]{Au})} Let $C$ be a finitely presented module. There exists a right almost split map with codomain $C$ \iff the module $C$ has got local endomorphism ring. Moreover, if $C$ is non-projective, then there is even an almost split sequence ending in $C$.
\end{thm}

\medskip

In what follows, given a set $X$, we denote by $|X|$ the cardinality of $X$. Formally, $|X|$ denotes the set of all ordinal numbers smaller than the cardinal number $|X|$ (which is an ordinal as well). For example, if $X$ is a finite set of $n$ elements, then $|X|=\{0,1,\dotsc,n-1\}=n$. As usual, $|X|+|Y|$ denotes the cardinality of the disjoint union of $X$ and $Y$. If $\alpha, \beta$ are ordinals, we use interchangeably the notations $\alpha<\beta, \alpha\in\beta$ in the obvious meaning `$\alpha$ is less than $\beta$'. Finally, for an infinite cardinal $\theta$, $\cf(\theta)$ denotes the \emph{cofinality of $\theta$},i.e., the least cardinality of a set of smaller ordinals which converge to $\theta$. We always have $\theta\geq\cf(\theta)$. An infinite cardinal $\theta$ is called \emph{regular} if $\theta = \cf(\theta)$, otherwise it is called \emph{singular}. Note that $\cf(\theta)$ is always a regular cardinal. By $\omega$ or $\aleph_0$, we denote the countable cardinal.

For any two sets $X,Y$, we denote by ${}^XY$ the set of all functions from $X$ to~$Y$. Moreover, if $\lambda, \mu$ are cardinals, then $\lambda^{<\mu}$ denotes the cardinality of the set ${}^{<\mu}\lambda = \{\eta:\alpha\to \lambda\mid \alpha <\mu, \eta\hbox{ a function}\}$. Similarly, $\lambda^\mu$ denotes the cardinality of the set~${}^\mu \lambda$. It comes in handy to view the elements of ${}^\mu\lambda$ as subsets of $\mu\times\lambda$.

\smallskip

For a module $M = \prod_{i\in I} M_i$ and an infinite cardinal $\mu$, we denote by $\prod_{i\in I}^{<\mu} M_i$ the submodule of $M$ consisting of all elements with support of cardinality $<\mu$. We call this submodule \emph{a $\mu$-bounded product} of the modules $M_i$.

We say that a module $M$ is \emph{finitely presented} if the functor $\Hom(M,-)$ commutes with direct limits. 
For an infinite cardinal $\theta$, we call a module $M$ \emph{$\theta$-presented} (\emph{$<\theta$-presented}, resp.) provided that $M$ is the direct limit of a direct system of cardinality $\leq\theta$ ($<\theta$, resp.) consisting of finitely presented modules.

We finish by recalling a useful classic result. We say that a well-ordered direct system of modules $(M_\alpha, m_{\beta\alpha} \mid \alpha\leq\beta<\mu)$, where $\mu$ is a regular infinite cardinal, is \emph{continuous} provided that, for each $\delta<\mu$ limit, $M_\delta = \varinjlim (M_\alpha, m_{\beta\alpha} \mid \alpha\leq\beta<\delta)$.

\begin{lem}\label{l:present} Let $\theta$ be an infinite cardinal and $M$ a $\theta$-presented module. Then $M$ is the direct limit of a continuous well-ordered direct system of cardinality $\cf(\theta)$ consisting of $<\theta$-presented modules.
\end{lem}

\begin{proof} For a suitable directed poset $(I,\leq)$, we express the module $M$ as the direct limit of a system $\mathcal F = (F_i, f_{ji}:F_i \to F_j \mid i\leq j\in I)$ consisting of finitely presented modules. Moreover, we can w.l.o.g.\ assume that $|I| = \theta$. If $\theta = \aleph_0$, there is a cofinal countable well-ordered subsystem of $\mathcal F$, so we can assume that $\theta$ is uncountable.

Let $(\beta _\gamma \mid \gamma<\cf(\theta))$ be a continuous increasing sequence of infinite ordinals smaller than $\theta$ converging to $\theta$. From this system, we easily build an $\subseteq$-increasing sequence $(I_\gamma \mid \gamma<\cf(\theta), |I_\gamma| = |\beta_\gamma|)$ of directed subposets of $(I,\leq)$ \st $I_\delta = \bigcup _{\gamma<\delta} I_\gamma$ for $\delta$ limit, and $I = \bigcup _{\gamma<\cf(\theta)} I_\gamma$ (cf. \cite[Lemma 1.6]{AR}).

It remains to define $M_\gamma = \varinjlim (F_i, f_{ji}:F_i \to F_j \mid i\leq j\in I_\gamma)$ and, for all $\gamma\leq\delta<\cf(\theta)$, let $m_{\delta\gamma}$ be the canonical colimit factoring map. Then $M_\gamma$ is $<\theta$-presented, for each $\gamma<\cf(\theta)$, and $M=\varinjlim (M_\gamma, m_{\delta\gamma} \mid \gamma\leq\delta<\cf(\theta))$.
\end{proof}


\section{Construction of tree modules}
\label{sec:tree}

Let $\theta$ be an infinite cardinal. We call a module $M$ \emph{finitely $\theta$-separable} if it is the directed union of a system consisting of $<\theta$-presented direct summands of $M$. We denote by $\mathcal S_\theta$ the class of all finitely $\theta$-separable modules.

In what follows, we construct a particular type of these modules, and we show that $\mathcal S_\theta$ forms a test class for splitting of epimorphisms with $\theta$-presented codomain. By this, we mean that, given any epimorphism $f:B\to C$ with $\theta$-presented codomain, there is a module $L\in\mathcal S_\theta$ such that $f$ is a split epimorphism \iff $\Hom_R(L,f)$ is surjective.

The construction presented in this section is not entirely new. It is an uncountable variant of a~well-known technique, cf. \cite[\S 5]{SlT}. However, we need to analyse the resulting object in more detail.

\smallskip

Recall that a module is \emph{pure-projective} if it is a direct summand in a direct sum of finitely presented modules. If a module $M$ has the property that each of its finite subsets is contained in a pure-projective submodule of $M$ which is pure in $M$, we call $M$ a \emph{Mittag--Leffler module}. We denote the class of all Mittag--Leffler modules by $\mathcal{ML}$. The modules in $\mathcal{ML}$ have been studied a lot in the last 40 years, and many equivalent characterizations are known of this class (cf. \cite[Theorem 3.14]{GT}). We have chosen this one since it immediately yields $\mathcal S_{\aleph_0}\subseteq\mathcal{ML}$.

\smallskip

As the first step towards the announced construction, we recall an instance of the classic result on embedding of direct limits into reduced products, cf. the proof of \cite[Theorem 3.3.2]{Pr}.

\begin{prop}\label{p:embed} Let $(C_\alpha,f_{\beta\alpha}:C_\alpha\to C_\beta\mid\alpha\leq\beta<\mu)$ be a well-ordered direct system of modules indexed by an infinite regular cardinal $\mu$. Then there is an embedding of pure short exact sequences
$$\begin{CD}
	  0			@>>> \prod_{\alpha<\mu}^{<\mu} C_\alpha @>{\subseteq_*}>> \prod_{\alpha<\mu} C_\alpha @ >>> \prod_{\alpha<\mu} C_\alpha/\prod_{\alpha<\mu}^{<\mu} C_\alpha @>>> 0\\	
	 @. 			@AAA 		@A{\rho_0}AA  @A{\sigma_0}AA  @.		\\
	0	@>{}>>	F  @>{\subseteq_*}>>  \bigoplus _{\alpha<\mu} C_\alpha  @>>> \varinjlim _{\alpha<\mu} C_\alpha @>>> 0
\end{CD}$$
where the second row is the canonical presentation of the direct limit, $\sigma_0$ is a pure monomorphism, and for all $\alpha<\mu$ and $x\in C_\alpha$, we have $\rho_0(x)(\beta)=f_{\beta\alpha}(x)$ if $\alpha\leq\beta<\mu$, and $\rho_0(x)(\beta)=0$ otherwise.
\end{prop}

Applications of the embedding above are not very frequent in the literature. As a starter, we show the following interesting test for whether a module is cotorsion, answering \cite[Question 33]{Be}. Recall that a module $M$ is \emph{cotorsion} provided that $\Ext_R^1(F,M) = 0$ whenever $F$ is a flat module.

\begin{thm}\label{t:Berg} Let $R$ be a countable unital ring with $R^\omega$ flat and Mittag--Leffler, and let us denote by $\iota:R^{(\omega)}\to R^\omega$ the pure inclusion. Then a module $M$ is cotorsion \iff $\Hom_R(\iota,M)$ is surjective.
\end{thm}

\begin{proof} First, if $M$ is cotorsion, $R^\omega/R^{(\omega)}$ flat yields $\Ext^1_R(R^\omega/R^{(\omega)},M) = 0$. Thus $\Hom_R(\iota,M)$ is onto.

Now assume that $M$ is not a cotorsion module. Since $R$ is countable, there exists a countably presented flat module $F$ such that $\Ext^1_R(F,M)\neq 0$. By Lazard's theorem, $F$ is the direct limit of a countable well-ordered system of free modules of finite rank. Hence, it is a pure submodule in $R^\omega/R^{(\omega)}$ by Proposition~\ref{p:embed}. Let $\sigma$ denote this pure embedding, and let $\pi:R^\omega \to R^\omega/R^{(\omega)}$ be the canonical (pure) epimorphism.

Forming the pullback of $\sigma$ and $\pi$,

$$\begin{CD} 
    @. @. 0 @. 0 @. \\
    @. @. @AAA  @AAA @.  \\
	  @. @. \Coker(\varepsilon) @= \Coker(\sigma) @. \\
    @. @.  @AAA @AAA  @. \\
	0 @>>> R^{(\omega)} @>{\iota}>>	R^\omega @>{\pi}>> R^\omega/R^{(\omega)} @>>> 0 	\\
	@. @|			@A{\varepsilon}AA	  @A{\sigma}AA   @. \\
	0 @>>> R^{(\omega)} @>{\subseteq}>>	N @>{}>> F  @>>> 0 \\
	@. @.			@AAA	  @AAA   @. \\
    @. @. 0 @. \hbox{ }0, @.
\end{CD}$$
we see that $N$ is isomorphic to a pure submodule of a flat Mittag--Leffler module, hence it is flat and Mittag--Leffler. Moreover, $N$ is countably presented, and so it is a projective module; in particular, $\Ext^1_R(N,M) = 0$. Since $\Ext^1_R(F,M)\neq 0$, there is a homomorphism $h:R^{(\omega)} \to M$ which cannot be extended to an element of $\Hom _R(N,M)$, from which it readily follows that there is no extension $R^\omega \to M$ of $h$ either.
\end{proof}

\begin{rem} The countable unital rings $R$ for which $R^\omega$ is flat and Mittag--Leffler are precisely the left coherent ones satisfying the additional condition that intersections of finitely generated left submodules of ${}_RR^{(\omega)}$ are finitely generated (cf. \cite[Theorem~4.7]{HT}). In particular, all countable left noetherian unital rings satisfy the hypothesis of Theorem~\ref{t:Berg}.

The hypothesis on the cardinality of the ring is necessary. As a counterexample, take $R = \End_K(K^{(\omega)})$ where $K$ is a field. This is a self-injective von Neumann regular ring which is not right hereditary. At the same time $R\cong R^\omega$, and so the projective dimension of $R^\omega/R^{(\omega)}$ is at most one. It follows that there exists a~non-cotorsion module $M$ such that $\Ext_R^1(R^\omega/R^{(\omega)},M) = 0$.
\end{rem}

\medskip

\noindent\emph{The construction.} Assume we are given a well-ordered direct system $\mathcal C = (C_\alpha,f_{\beta\alpha}:C_\alpha\to C_\beta \mid \alpha\leq\beta<\mu)$ indexed by an infinite regular cardinal $\mu$ (as in the statement of Proposition~\ref{p:embed}), and a cardinal $\lambda$ \st $\lambda^{<\mu} = \lambda$. Then we use a simple

\begin{lem}\label{l:bigtree} There is a $T\subset{}^\mu \lambda$ such that $|T| = \lambda^\mu$ and each two distinct elements $\eta,\zeta\in T$ coincide on an initial segment of $\mu$, i.e., $\Dom(\eta\cap\zeta)\in\mu$. \end{lem}

\begin{proof} We embed ${}^\mu\lambda$ into ${}^\mu({}^{<\mu}\lambda)$ via the assignment $\nu:\eta\mapsto (\alpha\mapsto \eta\restriction\alpha)$. Using the assumption on $\lambda$, we can fix a bijection $\iota:{}^{<\mu}\lambda\to \lambda$. We define $T$ as the set $\{\iota\circ(\nu(\eta)) \mid \eta\in {}^\mu\lambda\}$. Now, for two distinct $\eta,\zeta\in {}^\mu\lambda$, we consider the least $\alpha$ \st $\eta(\alpha)\neq\zeta(\alpha)$. From the definition of $\nu$, it follows that $\Dom(\nu(\eta)\cap\nu(\zeta)) = \alpha$, whence also $\Dom(\iota\circ(\nu(\eta))\cap\iota\circ(\nu(\zeta))) = \alpha$.
\end{proof}

Elements of the set $T$ we have obtained in this way form branches of length $\mu$ of the forest $\bigcup T$.\footnote{The partial order is defined by $(\alpha,\beta)<(\gamma,\delta) \Leftrightarrow \alpha < \gamma \,\&\, (\exists \eta\in T)(\eta(\alpha) = \beta \,\&\, \eta(\gamma) = \delta$).} We are going to decorate these branches uniformly using our well-ordered direct system $\mathcal C$. This is done via an enhancement of Proposition~\ref{p:embed}. Recall that we view the elements of $T$ as subsets of $\mu\times \lambda$.

Set $C=\varinjlim \mathcal C$. For all $(\alpha,\beta)\in\bigcup T$, put $C_{\alpha,\beta} = C_\alpha$. We have the following commutative diagram with exact rows (where $\pi$ denotes the canonical projection)

$$\begin{CD}
	  0			@>>> \prod_{(\alpha,\beta)\in\bigcup T}^{<\mu} C_{\alpha,\beta} @>{\subseteq_*}>> \prod_{(\alpha,\beta)\in\bigcup T} C_{\alpha,\beta} @ >\pi>> \Img(\pi) @>>> 0\\	
	 @. 			@AAA 		@A{\rho}AA  @A{\sigma}AA  @.		\\
	0	@>{}>>	F^{(T)}  @>{\subseteq_*}>>  (\bigoplus _{\alpha<\mu} C_\alpha)^{(T)}  @>{\tau}>> C^{(T)} @>>> 0
\end{CD}$$
where the bottom row is just a coproduct of the one in Proposition~\ref{p:embed}, and $\rho$ is defined as follows: for each $\eta\in T$, let $\nu_\eta:\bigoplus _{\alpha<\mu} C_\alpha \to (\bigoplus _{\alpha<\mu} C_\alpha)^{(T)}$ be the canonical embedding onto the $\eta$th coordinate. For $\alpha<\mu$ and $x\in C_\alpha$, we set $\rho\nu_\eta(x)=k$ where $k(\beta,\eta(\beta)) = f_{\beta\alpha}(x)$ if $\alpha\leq\beta<\mu$, and $k(\beta,\gamma)=0$ otherwise.

Notice that, for each $\eta\in T$, the diagram from Proposition~\ref{p:embed} embeds into the diagram above. In the second row, this is just the coproduct embedding corresponding to $\eta$. The first row embeds as follows (where $Q = \Img(\pi\restriction \prod_{\alpha<\mu} C_{\alpha,\eta(\alpha)})$):
$$\begin{CD}
	  0			@>>> \prod_{\alpha<\mu}^{<\mu} C_{\alpha,\eta(\alpha)} @>{\subseteq_*}>> \prod_{\alpha<\mu} C_{\alpha,\eta(\alpha)} @ >{\pi\restriction\prod_{\alpha<\mu} C_{\alpha,\eta(\alpha)}}>> Q @>>> 0. \end{CD}$$

Using this observation, from Proposition~\ref{p:embed}, we know that $\rho\nu_\eta$ is a monomorphism for each $\eta\in T$ (since $\rho_0$ is), and also that $\rho\nu_\eta\restriction F$ maps into the $\mu$-bounded product. Thus $\rho\restriction F^{(T)}$ maps into the $\mu$-bounded product as well, whence we get the unique $\sigma$ completing the diagram above.

In what follows, for a class $\mathcal X\subseteq\ModR$, $\hbox{Sum}(\mathcal X)$ denotes the class of all modules isomorphic to a direct sum of modules from $\mathcal X$.

\begin{lem}\label{l:technical} With the notation as above, we have:
\begin{enumerate}
	\item $\sigma$ is a monomorphism.
	\item There is an exact sequence $0 \longrightarrow D \overset{\subseteq}\longrightarrow L \overset{g}\longrightarrow C^{(T)} \longrightarrow 0$ where $L=\Img(\rho)$ and $D=\Ker(\pi\restriction L)$. Moreover, if $\lambda\geq |R|$, and the system $\mathcal C$ consists of $\lambda$-presented modules, then $|D|\leq\lambda$.
  \item For each finite subset $S$ of $T$, the module $L_S = \sum _{\eta\in S} \Img(\rho\nu_{\eta})$ is a direct summand in $L$, and $L_S\in\hbox{\rm Sum}(\{C_\alpha\mid\alpha<\mu\})$. Furthermore, we have $L = \bigcup\,\{L_S\mid S\subset T\hbox{ finite}\,\}$.
  \item For each $S\subseteq T$ with $|S| \leq \mu$, the module $L_S = \sum _{\eta\in S} \Img(\rho\nu_{\eta})$ decomposes as $L_S^\prime\oplus K_S^\prime$ where $g\restriction L_S^\prime = 0$ and $K_S^\prime\in\hbox{\rm Sum}(\{C_\alpha\mid\alpha<\mu\})$.

\end{enumerate}
\end{lem}

\begin{proof} $(1)$. Pick an arbitrary element $x = x_1+\dotsb+x_n\in (\bigoplus _{\alpha<\mu} C_\alpha)^{(T)}$ \st $0\neq x_i=\nu_{\eta_i}(y_i)$, for all $i=1,\dotsc, n$, where $\eta_1,\dotsc, \eta_n\in T$ are pairwise distinct. Assume that $\rho(x)$ belongs to the $\mu$-bounded product. Then there is a~ $\gamma<\mu$ such that $\eta_1(\gamma),\dotsc, \eta_n(\gamma)$ are pairwise distinct and $\rho(x)(\beta,\eta_i(\beta)) = 0$ for each $i = 1,\dotsc, n$ and $\beta\geq\gamma$. The former condition and the way $\rho$ is defined imply that the latter condition can be rephrased as $\rho\nu_{\eta_i}(y_i)(\beta,\eta_i(\beta))=0$ for each $i = 1,\dotsc, n$ and $\beta\geq\gamma$.

Since $\sigma_0$ from Proposition~\ref{p:embed} is a monomorphism, we infer (using the observation with the embedding of diagrams) that $y_i\in F$ for each $i = 1,\dotsc, n$. It follows that $x\in F^{(T)}$. Since $x$ was arbitrary, we conclude that $\sigma$ is a monomorphism.

\smallskip

The first part of $(2)$ follows from $(1)$ by putting $g = \sigma^{-1}(\pi\restriction L)$ where $\sigma^{-1}$ is the inverse of the isomorphism $\sigma:C^{(T)}\to \Img(\sigma)$. Let us prove the moreover clause.

Since $\mu\leq\lambda$ and $\bigcup T\subseteq \mu\times\lambda$, we have $|\bigcup T|\leq \lambda$. By the assumption, the cardinality of modules in the system $\mathcal C$ is at most $\lambda$. The assumption $\lambda = \lambda^{<\mu}$ then implies that the cardinality of the $\mu$-bounded product is at most $\lambda$, too. Thus $|D|\leq\lambda$, since $D$ is a submodule in the $\mu$-bounded product.

\smallskip

$(3)$. Let $S=\{\eta_i \mid i < |S|\}$ be a finite subset in $T$. Put $n = |S|$. For each $\zeta\in T\setminus S$, let $\psi(\zeta)$ denote the least ordinal $\beta<\mu$ \st $(\beta,\zeta(\beta))\not\in\bigcup S$. For $i< n$, we put $\psi(\eta_i) = \min\{\beta<\mu\mid (\beta,\eta_i(\beta))\not\in\bigcup_{j<i}\eta_j\}$. We claim that
$$L = \Biggl(\;\bigoplus_{\eta\in S}\rho\nu_{\eta}(\bigoplus _{\psi(\eta)\leq\alpha<\mu} C_\alpha)\Biggr) \oplus \sum _{\zeta\in T\setminus S}\rho\nu_\zeta(\bigoplus _{\psi(\zeta)\leq\alpha<\mu}C_\alpha).\eqno{(*)}$$

The independence of the summands is clear since they occupy different canonical direct summands in the product $\prod_{(\alpha,\beta)\in\bigcup T}C_{\alpha,\beta}$. Let us denote the second term of the decomposition $(*)$ by $K_S$.

We have to show that, for every $\zeta\in T, \alpha<\mu$ and $y\in C_\alpha$, the element $x = \rho\nu_\zeta(y)$ can be written as $x^\prime + \sum_{i< n}\rho\nu_{\eta_i}(y_i)$ where $x^\prime\in K_S$ and $y_i\in\bigoplus_{\psi(\eta_i)\leq\alpha<\mu} C_\alpha$ for all $i<n$. This is clear if $\alpha\geq\psi(\zeta)$. Otherwise, we define an increasing (finite) sequence in $n=\{0,1,\dotsc,n-1\}$ as follows:

\smallskip

Set $m_0 = \min\{j<n\mid \eta_j(\alpha)=\zeta(\alpha)\}$. If $m_i$ is defined and $\beta_i = \Dom(\eta_{m_i}\cap\zeta)<\psi(\zeta)$, put $m_{i+1} = \min\{j<n \mid \eta_j(\beta_i)=\zeta(\beta_i)\}$. Set $\beta_{-1} = \alpha$, and let $k$ be the greatest \st $m_k$ is defined. Notice that $\psi(\eta_{m_{i+1}})=\beta_i$ for $0\leq i<k$, and $\psi(\eta_{m_0})\leq\alpha$. From the definition of $\rho$, it readily follows that

$$x = \rho\nu_\zeta(f_{\psi(\zeta)\alpha}(y))+\sum _{i=0}^k\rho\nu_{\eta_{m_i}}(f_{\beta_{i-1}\alpha}(y)-f_{\beta_i\alpha}(y)).$$
Note that the first summand is the $x^\prime$ we have looked for if $\zeta\not\in S$; otherwise, $x^\prime = 0$.

Having proved that $L$ decomposes as in $(*)$, we have also shown that $L_S$ is the first term of this decomposition. Thus $L_S$ is a direct summand in $L$ and $K_S$ is its complement. Moreover, since $\rho\nu_\eta$ is a monomorphism for each $\eta\in T$, we can conclude that

$$L_S = \bigoplus_{\eta\in S}\rho\nu_{\eta}(\bigoplus _{\psi(\eta)\leq\alpha<\mu} C_\alpha) \cong \bigoplus_{\eta\in S}\Biggl(\;\bigoplus _{\psi(\eta)\leq\alpha<\mu} C_\alpha\Biggr).$$

\smallskip

$(4)$. Set $\kappa = |S|$ and fix a (non-repeating) enumeration $S = \{\eta_\gamma \mid \gamma<\kappa\}$. For each $\gamma<\kappa$, let $\psi(\gamma)$ denote the least $\beta<\mu$ such that $(\beta,\eta_\gamma(\beta))\not\in\bigcup _{\delta<\gamma}\eta_\delta$. This definition is possible since $\mu$ is a regular cardinal and $\kappa\leq\mu$.

Let $E = \{k\in\prod_{(\alpha,\beta)\in\bigcup S} C_{\alpha,\beta} \mid (\forall \gamma<\kappa)(\forall \delta >\psi(\gamma))\,k(\delta,\eta_\gamma(\delta))=0\}$. Then $L_S^\prime = E\cap L_S$ is contained in the $\mu$-bounded product, hence $g\restriction L_S^\prime = 0$.

On the other hand, the module $$K_S^\prime = \bigoplus_{\gamma<\kappa}\rho\nu_{\eta_\gamma}(\bigoplus _{\psi(\gamma)<\alpha<\mu} C_\alpha) \cong \bigoplus_{\gamma<\kappa}\Biggl(\;\bigoplus _{\psi(\gamma)<\alpha<\mu} C_\alpha\Biggr)$$
is a complement of $L_S^\prime$ in $L_S$. Indeed, we readily check that $L_S^\prime\cap K_S^\prime = 0$, and for each $\gamma<\kappa$, $\alpha\leq\psi(\gamma)$ and $y\in C_\alpha$, we can write $\rho\nu_{\eta_\gamma}(y) = \rho\nu_{\eta_\gamma}(y-f_{(\psi(\gamma)+1)\alpha}(y))+\rho\nu_{\eta_\gamma}(f_{(\psi(\gamma)+1)\alpha}(y))$ where the first term is in $L_S^\prime$ and the second in $K_S^\prime$.
\end{proof}

\begin{rem}\label{r:pure} By the construction, $\tau = g\rho$. It follows that $g$ is a pure epimorphism. In fact, it is even $\mu$-pure since $\tau$ is, i.e., any homomorphism from a $<\mu$-presented module into $C^{(T)}$ factorizes through $\tau$ (and hence through $g$).
\end{rem}

The module $L$ from the short exact sequence in Lemma~\ref{l:technical}$(2)$ is the tree module constructed from the data $\mathcal C, \lambda, T$. Choosing this data a little bit more carefully, we can impose further properties on $L$ and the short exact sequence it fits in.

\begin{prop}\label{p:further} With the notation as above. Let $\theta$ be an infinite cardinal with $\cf(\theta)=\mu$. Assume that $\mathcal C$ consists of $<\theta$-presented modules. Then the following hold.
\begin{enumerate}
	\item The module $L$ is finitely $\theta$-separable.
  \item If $e\in\End_R(L)$ is an idempotent with $e\restriction D = 0$, and $\End_R(\Img(e))$ is a~local ring, then $\Img(e)$ is $<\theta$-presented.
\end{enumerate}
\end{prop}

\begin{proof} $(1)$. By Lemma~\ref{l:technical}$(3)$, the module $L$ is the directed union of the system $(L_S\mid S\subset T\hbox{ finite})$. We replace this system by $(L_{S,\beta} \mid \gamma_S\leq\beta<\mu, S\subset T\hbox{ finite})$ defined as follows:

Let a finite $S\subset T$ be fixed, and let $\psi: T\to \mu$ be defined as in the proof of Lemma~\ref{l:technical}$(3)$. Put $\gamma _S = \max\{\psi(\eta)\mid \eta\in S\}$ and

$$L_{S,\beta} = \bigoplus_{\eta\in S}\rho\nu_{\eta}(\bigoplus _{\psi(\eta)\leq\alpha\leq\beta} C_\alpha) \cong \bigoplus_{\eta\in S}\Biggl(\;\bigoplus _{\psi(\eta)\leq\alpha\leq\beta} C_\alpha\Biggr).$$
This is a direct sum of $<\cf(\theta)$ of $<\theta$-presented modules, hence it is $<\theta$-presented. Moreover, $L_{S,\beta}$ is also a direct summand in $L_S$ with the complement $K_{S,\beta}$ where
$$K_{S,\beta}=\bigoplus_{\eta\in S}\rho\nu_{\eta}(\bigoplus _{\beta<\alpha<\mu} C_\alpha)\cong (\bigoplus _{\beta<\alpha<\mu} C_\alpha)^{S}.$$
We have proved $(1)$, since $L_S$ splits in $L$ by Lemma~\ref{l:technical}$(3)$.

\smallskip

For $(2)$, put $H = \Img(e)$, and let $h:C^{(T)}\to H$ be the epimorphism \st $hg = e$. We can assume that $H\neq 0$; otherwise the conclusion trivially holds.

Pick any non-zero $u\in H$. Then there is a finite subset $S$ of $T$ \st $e(u)=h\varepsilon_S\pi_Sg(u) = u$ where $\pi_S$ and $\varepsilon_S$ are the canonical projection, embedding resp., between $C^{(T)}$ and $C^S$. Using that $\End_R(H)$ is local, we see that $h\varepsilon_S\pi_Sg\restriction H$ is an automorphism of $H$. Thus we can w.l.o.g.\ assume that $e$ factorizes through $\pi_Sg$.


From Lemma~\ref{l:technical}$(3)$, we know that $L = L_S\oplus K_S$ where $K_S\subseteq \sum_{\zeta\in T\setminus S}\Img(\rho\nu_\zeta)$. We readily check that $K_S\subseteq\Ker(\pi_Sg)\subseteq\Ker(e)$. It follows that id${}_H = e\restriction H$ factorizes through $L/K_S\cong L_S$ which is a direct sum of $<\theta$-presented modules. So $H$ is a direct summand in a direct sum of $<\theta$-presented modules. However, $\End_R(H)$ is local whence $H$ must be $<\theta$-presented itself.
\end{proof}

\begin{rem}\label{r:rel} Fix a class $\mathcal F$ of finitely presented modules \st $\mathcal F$ is closed under finite direct sums. We can relativize the construction in this section to the subcategory $\varinjlim \mathcal F$ as follows.

Assume that $C\in\varinjlim \mathcal F$. By \cite[Lemma 2.13]{GT}, $\varinjlim \mathcal F$ is closed under taking direct limits and pure submodules. It follows that, once we choose the modules $C_\alpha, \alpha<\mu$, from the class $\varinjlim\mathcal F$ (which can be easily done, cf. the proof of Lemma~\ref{l:present}), the modules $D$ and $L$ belong to $\varinjlim\mathcal F$ as well---use Remark~\ref{r:pure} and Lemma~\ref{l:technical}$(3)$.
\end{rem}

\section{The main theorem}
\label{sec:main}

We start with a general observation.

\begin{lem}\label{l:Hun} Let $\lambda$ be an infinite cardinal, $0\rightarrow D \rightarrow L \overset{g}\rightarrow C^{(2^\lambda)} \rightarrow 0$ a short exact sequence, and $f:B \to Y$ be an epimorphism. Assume that $|\Ker(f)|\leq 2^\lambda$ and $|D|\leq\lambda$. Then the following holds.
$$\Hom_R(C,f)\hbox{ is onto \iff }\Img(\Hom_R(g,Y))\subseteq \Img(\Hom_R(L,f)).$$

Moreover, the group $\,\Img(\Hom_R(g,Y))\, / \Img(\Hom_R(L,f))\cap \Img(\Hom_R(g,Y))$ is either trivial or of cardinality $\geq 2^{2^\lambda}$.
\end{lem}

\begin{proof} Put $A = \Ker(f)$. We have the following commutative diagram with exact rows and columns:

$$\begin{CD}
		0				@>>> \Hom_R(C^{(2^\lambda)},B) @>{\Hom_R(g,B)}>> \Hom_R(L,B) \\
		@.			@V{\Hom_R(C^{(2^\lambda)},f)}VV		@V{\Hom_R(L,f)}VV 	\\
	  0				@>>> \Hom_R(C^{(2^\lambda)},Y) @>{\Hom_R(g,Y)}>> \Hom_R(L,Y) \\	
	 @. 			@V{\varepsilon}VV 		@V{\xi}VV		\\
	\Hom_R(D,A)	@>{\delta}>>	\Ext_R^1(C^{(2^\lambda)},A)  @>{\Ext_R^1(g,A)}>> \Ext_R^1(L,A).
\end{CD}$$

Suppose that $\Hom_R(C,f)$ is surjective. Then $\Hom_R(C^{(2^\lambda)},f)$ is surjective as well, and we get $\;\Img(\Hom_R(g,Y))\subseteq \Img(\Hom_R(L,f))$ using the commutativity of the upper rectangle.

If $\Hom_R(C,f)$ is not surjective, then $\Img(\varepsilon)\cong(\Hom_R(C,Y)/\Img(\Hom_R(C,f)))^{2^\lambda}$ implies $|\Img(\varepsilon)|\geq 2^{2^\lambda}$. On the other hand, $2^\lambda=(2^\lambda)^\lambda \geq |A|^{|D|}\geq|\Hom_R(D,A)|\geq|\Img(\delta)|$. Using this cardinality discrepancy, the exactness of the third row and the commutativity of the lower rectangle, we obtain a set $W\subseteq \Hom_R(C^{(2^\lambda)},Y)$ of cardinality $2^{2^\lambda}$ such that the restriction of the map $\Ext_R^1(g,A)\varepsilon= \xi\Hom_R(g,Y)$ to $W$ is one-one. It follows that
$$\Img(\xi\restriction\Img(\Hom_R(g,Y)))\cong \Img(\Hom_R(g,Y))\, / \Img(\Hom_R(L,f))\cap \Img(\Hom_R(g,Y))$$
is a group of cardinality at least $2^{2^\lambda}$.
\end{proof}

The part $(1)$ in the following theorem says that the class, $\mathcal S_\theta$, of all finitely $\theta$-separable modules is a test class for splitting of epimorphisms with $\theta$-presented codomain. The second part constitutes the core of the proof of our main result.

\begin{thm} \label{t:factor} Let $\theta$ be an infinite cardinal and $f:B\to C$ be an epimorphism where $C$ is a $\theta$-presented module. Then the following hold.
\begin{enumerate}
	\item There exists $L\in\mathcal S_\theta$ such that $\Hom_R(L,f)$ is onto \iff $f$ splits.
	\item If $f$ is a right almost split map, then $C$ is $<\theta$-presented.
\end{enumerate}
\end{thm}

\begin{proof} Let us denote $\mu = \cf(\theta)$. We have a short exact sequence

$$0 \longrightarrow A \overset{m}\longrightarrow B \overset{f}\longrightarrow C \longrightarrow 0$$
where $C$ is the direct limit of a well-ordered direct system $\mathcal C = (C_\alpha, f_{\beta\alpha}:C_\alpha\to C_\beta\mid \alpha\leq\beta<\mu)$ consisting of $<\theta$-presented modules (cf. Lemma~\ref{l:present}).

\smallskip

Let $\lambda$ be an infinite cardinal \st $|R|+\theta\leq\lambda = \lambda ^{<\mu}<\lambda^{\mu}=2^\lambda\geq |A|$. For the cardinals $\lambda, \theta$ and the system $\mathcal C$, we use Lemma~\ref{l:technical} and Proposition~\ref{p:further} to obtain a short exact sequence
$$0 \longrightarrow D \overset{\subseteq}\longrightarrow L \overset{g}\longrightarrow {C^{(2^\lambda)}} \longrightarrow 0$$
with $|D|\leq\lambda$ and $L\in\mathcal S_\theta$ (recall that $|T|=\lambda^\mu=2^\lambda$). 

Proving $(1)$, we have the trivial implications: $f$ splits $\Rightarrow$ $\Hom_R(L,f)$ is onto $\Rightarrow \Img(\Hom_R(g,C))\subseteq \Img(\Hom_R(L,f))$. Using Lemma~\ref{l:Hun} with $C = Y$, we see that the last inclusion implies that $f$ splits.

Meanwhile, the assumption in $(2)$ implies that $f$ does not split. By the preceding paragraph, it follows that there is a $d\in\Hom_R(C^{(2^\lambda)},C)$ such that $dg$ does not factorize through~$f$. Using that $f$ is right almost split, we deduce that $dg$ is a split epimorphism.
By Proposition~\ref{p:few}$(1)$, $C$ has got local endomorphism ring. If we denote by $e\in\End_R(L)$ an idempotent which factorizes through $dg$, then $C\cong \Img(e)$ is a $<\theta$-presented module by Proposition~\ref{p:further}$(2)$.
\end{proof}

\begin{rem} As the cardinal $\lambda$ in the proof, we can pick for instance $\beth _{\mu}(|R|+|A|+\theta)$. Here, $\beth$ (Beth) is the function defined for all cardinals $\kappa$ inductively by putting $\beth_0(\kappa) = \kappa$, $\beth_{\alpha + 1}(\kappa) = 2^{\beth_\alpha(\kappa)}$ and $\beth_\alpha(\kappa) = \sum_{\beta<\alpha}\beth_{\beta}(\kappa)$ for $\alpha$ limit.
\end{rem}

The main scope of application of Theorem~\ref{t:factor}$(1)$ is in proving that a particular class of modules is not precovering. Recall, that a class $\mathcal B$ of modules is a \emph{precovering class} if, for any $M\in\ModR$, there exist $B\in\mathcal B$ and $f\in\Hom_R(B,M)$ such that, for all $B^\prime\in\mathcal B$, the map $\Hom_R(B^\prime,f)$ is surjective. The homomorphism $f$ is then called a \emph{$\mathcal B$-precover} of $M$.

\begin{cor}\label{c:precov} Let $\mathcal B$ be a precovering class of modules closed under direct summands and $\mathcal F$ a class of finitely presented modules closed under finite direct sums. Assume that $\mathcal B$ contains a generator, and that there exists an infinite cardinal $\theta$ \st $\mathcal S_\theta\cap\varinjlim\mathcal F\subseteq\mathcal B$. Then $\mathcal B$ contains all $\theta$-presented modules from $\varinjlim \mathcal F$. \end{cor}

\begin{proof} Since $\mathcal B$ contains a generator, any $\mathcal B$-precover must be an epimorphism. It remains to use Theorem~\ref{t:factor}$(1)$ together with Remark~\ref{r:rel}.
\end{proof}

If the ring $R$ is not right perfect, then there exists a countably presented flat module which is not Mittag--Leffler (by the famous result of H.\ Bass). By Corollary~\ref{c:precov} and the fact that $\mathcal S_{\aleph_0}\subseteq\mathcal{ML}$, we readily see, taking for $\mathcal F$ the class of all free modules of finite rank, that the class of all flat Mittag--Leffler modules is not precovering in this case (cf. \cite[\S 3]{AST}). By a similar argument, we get that the class $\mathcal{ML}$ is not precovering unless $R$ is right pure semisimple (where the equality $\mathcal{ML} = \ModR$ holds). The point is that over a ring $R$ which is not right pure-semisimple, there exists a countably presented module which is not Mittag--Leffler, cf. \cite[Lemma 5.1]{AST}.

\smallskip

We can prove the main result of our paper.

\begin{thm}\label{t:main} Let $R$ be a ring and $C$ be a module. Then $C$ is a codomain of a~right almost split map \iff $C$ is a finitely presented module with local endomorphism ring. \end{thm}

\begin{proof} The if part follows from Theorem~\ref{t:Aus}. Assume that $f:B\to C$ is a right almost split map. Then $C$ has got local endomorphism ring by Proposition~\ref{p:few}$(1)$. By the part $(2)$ of the same proposition, $f$ non-surjective yields $C$ projective. However, $C$ projective and $\End_R(C)$ local immediately imply that $C$ is finitely presented (it~is even a direct summand in the regular module $R_R$).

If $f$ is an epimorphism and $C$ is not finitely presented, then there is a least infinite cardinal $\theta$ such that $C$ is $\theta$-presented. We use Theorem~\ref{t:factor}$(2)$ to get the contradiction.
\end{proof}

\begin{rem} \label{r:categ} By \cite{Kr}, the if part of Theorem~\ref{t:main} holds in the general setting of \emph{finitely accessible additive categories} (also called locally finitely presented additive categories in \cite{DG}). So does the only-if part: indeed, by \cite[Theorem 1.1]{DG}, any finitely accessible additive category is equivalent to the category of flat modules over a ring with enough idempotents. By Remark~\ref{r:rel}, choosing $\mathcal F$ as the class of all finitely generated projective modules, we know that our construction relativizes to the category Flat-$R$.
\end{rem}

Theorem~\ref{t:main} has an immediate consequence also for the first term of any almost split sequence. Recall that a module is called \emph{pure-injective} if it is injective relative to pure embeddings.

\begin{cor} Let $0\to A \overset{m}\rightarrow B \to C \to 0$ be an almost split sequence. Then $C$ is finitely presented and $A$ is pure-injective.
\end{cor}

\begin{proof} The first part trivially follows from Theorem~\ref{t:main}. Assume that $A$ is not pure-injective. Then the pure embedding of $A$ into its pure-injective envelope does not split, and hence factorizes through $m$. It follows that the almost split sequence is pure. However, it would split in such a case since $C$ is finitely presented.
\end{proof}

\section{Morphisms determined by objects}
\label{sec:Aus}

In his famous paper \cite{A}, M.\ Auslander studied closely the general notion of a~morphism determined by object.

\begin{defn}\label{d:det} Let $C$ be a module and $f:B\to Y$ a homomorphism. We say that \emph{$f$ is right $C$-determined} if the following holds:

For any $B^\prime\in\ModR$ and $h\in\Hom_R(B^\prime, Y)$, the map $h$ factorizes through $f$ \iff $\,\Img(\Hom_R(C,h))\subseteq\Img(\Hom_R(C,f))$.
\end{defn}

Of course, the direct implication in Definition~\ref{d:det} always holds. The non-trivial and highly restrictive part is the converse. It turns out that the notion of a right almost split map is just a special case of this concept where $C = Y$, $\End_R(C)$ is local and $\Img(\Hom_R(C,f))$ is the Jacobson radical of $\End_R(C)$, cf. \cite[\S II.2]{A}.

\smallskip

Using the machinery developed, we can prove the following theorem. Recall that, given an infinite cardinal $\theta$, an epimorphism $f:B\to Y$ is called \emph{$\theta$-pure} provided that any homomorphism from a $<\theta$-presented module into $Y$ factorizes through~$f$. Thus the notion of $\aleph_0$-pure epimorphism coincides with the usual notion of pure epimorphism.

\begin{thm}\label{t:pure} Let $\theta$ be an infinite regular cardinal and $C$ a $\theta$-presented module. Assume that $f^\prime:B\to Y^\prime$ is a right $C$-determined map. Let $f:B\to \Img(f^\prime)$ denote the epimorphism which coincides with $f^\prime$. Then $f$ is not a $\theta$-pure epimorphism unless $f$ splits.
\end{thm}

\begin{proof} Assume that $f$ is a non-split epimorphism. We aim to prove that it is not $\theta$-pure. For this, set $\mu = \cf(\theta) = \theta, A = \Ker(f)$ and $Y = \Img(f)$. Fix a well-ordered direct system $\mathcal C = (C_\alpha, f_{\beta\alpha}:C_\alpha\to C_\beta\mid \alpha\leq\beta<\mu)$ consisting of $<\theta$-presented modules \st $C = \varinjlim \mathcal C$. As in Theorem~\ref{t:factor}, let $\lambda$ be an infinite cardinal \st $|R|+\theta\leq\lambda = \lambda ^{<\mu}<\lambda^{\mu}=2^\lambda\geq |A|$, and
$$0 \longrightarrow D \overset{\subseteq}\longrightarrow L \overset{g}\longrightarrow {C^{(2^\lambda)}} \longrightarrow 0$$
be a short exact sequence with $|D|\leq\lambda$ and $L\in\mathcal S_\theta$, obtained for the data $\lambda, \theta, \mathcal C$ by Lemma~\ref{l:technical} and Proposition~\ref{p:further}.

Observe that the epimorphism $f$ is right $C$-determined, too. Using this and our assumption that $f$ does not split, we get that $\Hom_R(C,f)$ is not onto. By Lemma~\ref{l:Hun}, we obtain an element $d\in\Hom_R(C^{(2^\lambda)},Y)$ \st the map $dg$ does not factorize through $f$. Again, since $f$ is right $C$-determined, it follows that there is an $h\in\Hom_R(C,L)$ such that $dgh$ does not factorize through $f$.

Since $C$ is $\theta$-generated, there is a set $S\subseteq T$ of cardinality at most $\theta (= \mu)$ such that $\Img(h)\subseteq L_S = \sum _{\eta\in S} \Img(\rho\nu_\eta)$. Let $L_S = L_S^\prime\oplus K_S^\prime$ be the decomposition from Lemma~\ref{l:technical}$(4)$. We can write $g = g_1g_0$, where $g_0:L\to L/L_S^\prime$ is the canonical projection. Then $\Img(g_0h)\subseteq L_S/L_S^\prime\cong K_S^\prime$.

Since $dg_1g_0h$ does not factorize through $f$, neither does $dg_1\restriction L_S/L_S^\prime$. However, $L_S/L_S^\prime$ is isomorphic to a direct sum of $<\theta$-presented modules, and so $f$ cannot be $\theta$-pure.
\end{proof}

Let us record an immediate corollary of the above result.

\begin{cor}\label{c:count} Let $C$ be a countably presented module and $f:B\to Y$ be a right $C$-determined homomorphism. Then $f:B\to \Img(f)$ either splits or is not a pure epimorphism. In particular, $f:B\to\Img(f)$ splits whenever $R$ has got weak global dimension $\leq 1$ and $Y$ is a flat module.
\end{cor}
\medskip

\noindent\textbf{Acknowledgements.} I would like to express my gratitude to Dolors Herbera for inviting me to participate on the project MTM2011-28992-C02-01 of DGI MINECO (Spain). Significant part of the material in this paper was written during my stay at Universitat Aut\`onoma de Barcelona from 20~ March to 1 April 2015.

Many thanks also to Jan Trlifaj and Ivo Herzog for reading and discussing vast majority of the text.


\end{document}